\documentclass[final]{amsart}
\usepackage{amssymb}
\usepackage{hhline}
\usepackage{graphicx}
\usepackage{stmaryrd}
\usepackage{float}
\usepackage{cite}

\usepackage[utf8]{inputenc}
\usepackage{listings}
\usepackage{xcolor}

\newtheorem{theorem}{Theorem}[section]

\newtheorem{proposition}[theorem]{Proposition}
\newtheorem{corollary}[theorem]{Corollary}

\newtheorem{problem}[theorem]{Problem}
\theoremstyle{definition}

\newtheorem{example}[theorem]{Example}

\newcommand{\Aut}{\mathrm{Aut\mkern 2mu}}

\title{Note on the number of doppelsemigroups of small order}

\author{Volodymyr M. Gavrylkiv}
\address[V.~Gavrylkiv]{Vasyl Stefanyk Carpathian National University, Ivano-Frankivsk, Ukraine} \email{vgavrylkiv@gmail.com}

\begin{document}

\begin{abstract}
We study doppelsemigroups, i.e., algebraic structures equip\-ped with two associative binary operations satisfying a specified system of axioms. We investigate duality and isomorphisms of doppelsemigroups and examine the relationships between commutative, abelian, strong, and rectangular doppelsemigroups. Several examples are constructed, including nontrivial iso-opposite doppelsemigroups, noncommutative iso-dual doppelsemigroups, nonabelian iso-cross-dual doppelsemigroups, and nonstrong rectangular iso-opposite doppelsemigroups. Furthermore, we refine the complete classification of nonisomorphic doppelsemigroups of order~3. Finally, we present computer-assisted calculations yielding the numbers of all pairwise nonisomorphic doppelsemigroups and strong doppelsemigroups of orders up to~$5$, as well as all pairwise nonisomorphic commutative, abelian, and rectangular doppelsemigroups of orders up to~$6$, obtained using \texttt{GAP}, \texttt{Python}, and \texttt{C++}.
\end{abstract}

\maketitle

\section*{Introduction}

Given a semigroup $(S, \dashv)$, consider a semigroup $(S, \vdash)$ defined on the same set. 
We say that the semigroups $(S, \vdash)$ and $(S, \dashv)$ are {\em interassociative} provided
$$(x\dashv y)\vdash z=x\dashv(y\vdash z)\ \text{  and  }\ (x\vdash y)\dashv z=x\vdash(y\dashv z)$$
for all $x, y, z\in S$.
When this occurs,  $(S, \vdash)$ is said to be an {\em interassociate} of  $(S, \dashv)$, or that the semigroups are interassociates of each other.
The present concept of interassociative semigroups originated in 1986 in Drouzy~\cite{Dr}, where it is noted that every group is isomorphic to each of its interassociates. J. B. Hickey in 1983~\cite{Hi83} dealt with the special case of interassociativity in which the operation $\vdash$ is defined by specifying $a\in S$ and stipulating that $x \vdash y = x\dashv a\dashv y$ for all $x, y\in S$. Clearly $(S, \vdash)$, which Hickey calls a {\em variant} of $(S, \dashv)$, is a semigroup that is an interassociate of $(S, \dashv)$. Methods of constructing interassociates were developed, for semigroups in general and for specific classes of semigroups, in 1997 by Boyd, Gould and Nelson~\cite{BoGN}. 

This paper is devoted to study of doppelsemigroups which are sets with two associative binary operations satisfying axioms
of interassociativity.
More accurately,
a {\em doppelsemigroup} is an algebraic structure $(D,\dashv,\vdash)$ consisting of a set $D$ equipped with two asso\-ciative binary operations $\dashv$ and $\vdash$ satisfying the following axioms:
\begin{align*}
(x \dashv y) \vdash z = x \dashv (y \vdash z), \hspace{10mm}(D_1) \\
(x \vdash y) \dashv z = x \vdash (y \dashv z). \hspace{10mm}(D_2)
\end{align*}
Thus, we can see that for any doppelsemigroup $(D,\dashv,\vdash)$, the semigroups $(D,\vdash)$ and $(D,\dashv)$ are interassociative, and conversely, if a semigroup $(D,\vdash)$ is an
interassociate of a semigroup $(D,\dashv)$, then $(D,\dashv,\vdash)$ and $(D,\vdash,\dashv)$ are doppelsemigroups. If $(D,\dashv,\vdash)$ is a doppelsemigroup, then rearranging the parentheses in an expression that contains only operations $\vdash$, $\dashv$ and elements of $D$ do not change the result.
A doppelsemigroup $(D,\dashv,\vdash)$ is called {\em commutative}~\cite{Zh2017AU} if both semigroups $(D,\dashv)$ and $(D,\vdash)$ are commutative. It is said to be {\em strong}~\cite{Zh2018} when the axiom $x\dashv (y\vdash z) = x\vdash (y\dashv z)$ holds. Moreover, $(D,\dashv,\vdash)$ is called {\em rectangular} if both $(D,\dashv)$ and $(D,\vdash)$ are rectangular semigroups in the sense that $x\dashv y \dashv z = x \dashv z$ and $x\vdash y \vdash z = x \vdash z$ for all $x,y,z \in D$.

The investigation of doppelsemigroups was initiated by A.~Zhuchok~\cite{Zh2017AU}. The concept of a doppelsemigroup originates from the study of dimonoids in the sense of Loday~\cite{L}. Linear analogues of doppelsemigroups, known as doppelalgebras, were introduced by Richter~\cite{R} in the context of algebraic $K$-theory, and commutative dimonoids provide natural examples of doppelsemigroups. Thus, the theory of doppelsemigroups is closely connected to both doppelalgebras and dimonoids. Moreover, a doppelsemigroup can be defined via the notion of a duplex~\cite{P}. Doppelsemigroups are also related to bisemigroups studied by Schein~\cite{Sh1}, which have applications in the theory of binary relations~\cite{Sh2}. Every semigroup $(D,\dashv)$ can be naturally regarded as a doppelsemigroup $(D,\dashv,\dashv)$, referred to as the \emph{trivial doppelsemigroup}. 

Several classes of doppelsemigroups have been explored extensively by A.~Zhuchok and his collaborators. Constructions include the free product of doppelsemigroups, the free (strong) doppelsemigroup, the free commutative (strong) doppelsemigroup, the free $n$-nilpotent (strong) doppelsemigroup, the free rectangular doppelsemigroup, and the free abelian doppelsemigroup~\cite{Zh2017AU, Zh2018, ZhK2018, ZhZhK2020}. Relatively free doppelsemigroups were investigated in~\cite{Zh2018M}, while the free $n$-dinilpotent (strong) doppelsemigroup was constructed in~\cite{ZhD2016, Zh2018}. In addition, A.~Zhuchok described the free left $n$-dinilpotent doppelsemigroup in~\cite{Zh2017}. Representations of ordered doppelsemigroups via binary relations were studied by Yu.~Zhuchok and J.~Koppitz~\cite{ZhK2019}.

In \cite{GR1}, all pairwise nonisomorphic doppelsemigroups of order $2$ were classified. There exist eight two-element doppelsemigroups, six of which are commutative, and all of which are strong. It was also found $75$ nonisomorphic three-element doppelsemigroups, among which $41$ are commutative and $65$ are strong.

Cyclic doppelsemigroups were investigated in \cite{GDS2}. A doppelsemigroup  $(G,\dashv,\vdash)$ is called a \emph{group doppelsemigroup} if $(G,\dashv)$ is a group, in which case $(G,\vdash)$ is isomorphic to $(G,\dashv)$. Such a doppelsemigroup is said to be \emph{cyclic} if $(G,\dashv)$ is a cyclic group. Up to isomorphism, there exist $\tau(n)$ finite cyclic (strong) doppelsemigroups of order $n$, where $\tau$ denotes the divisor-counting function, and there are infinitely many pairwise nonisomorphic infinite cyclic (strong) doppelsemigroups. Several functorial constructions and extensions of doppelsemigroups were developed in \cite{GUDS, GSDS, GLDS}.

In \cite{GDS_die}, rectangular doppelsemigroups were studied. In particular, com\-mutative rectangular doppelsemigroups were classified up to isomor\-phism, and combinatorial  results were obtained for the numbers of nonisomorphic rectangular  ideal extensions of left (right) zero doppelsemigroups by null quotients.

The present paper investigates duality and isomorphisms of doppelsemigroups, and examines the relationships between commutative, abelian, strong, and rectangular doppelsemigroups. Furthermore, we refine the complete classification of nonisomorphic doppelsemigroups of order~3. Finally, computer-assisted calculations provide the numbers of all pairwise nonisomorphic doppelsemigroups and strong doppelsemigroups of orders up to~5, as well as all pairwise nonisomorphic commutative, abelian, and rectangular doppelsemigroups of orders up to~6. Analogous results for dimonoids and $g$-dimonoids were obtained in \cite{Gdim1, Gdim2, Ggdim1, Gdim3}, where the automorphism groups of these algebraic structures were also explicitly determined.

\section{Preliminaries on semigroups}

An element $e$ of a semigroup $(S,*)$ is called  {\em a left identity} (resp.  {\em a right identity}) in $S$ if $e*a=a$ (resp. $a*e=a$) for any $a\in S$. An element $1$  is called  {\em an identity} if $1$ is a left identity and a right identity. An element $e$ of a semigroup $(S,*)$ is called  {\em a middle identity} in $S$ if $l*e*r=l*r$  for all $l,r\in S$. It is clear that each left (right) identity is a middle identity.

\smallskip

An element $z$ of a semigroup $S$ is called  {\em a left zero} (resp.  {\em a right zero}) in $S$ if $z*a=z$ (resp. $a*z=z$) for any $a\in S$. An element $0$  is called  {\em a  zero} if $0$ is a left zero and a right zero.

\smallskip

A semigroup $(S,*)$ is called a {\em null semigroup} if there exists an element $0\in S$ such that $x*y=0$ for all $x,y\in S$. In this case  $0$ is a zero of $S$.  By $O_{S^0}$ we denote a null semigroup with zero $0$ on a set $S$.  The null semigroups $O_{S^0}$ and $O_{T^z}$ are isomorphic if and only if $|S|=|T|$. If $S$ is a set of cardinality $|S|=n$, we use the notation $O_n$ for a representative of the class of semigroups isomorphic to $O_{S^0}$.

\smallskip

A semigroup $(S,*)$ is said to be a {\em left} (resp. {\em right}) {\em zero semigroup} if $a*b=a$ (resp. $a*b=b$) for any $a,b\in S$. By  $LO_S$ and $RO_S$ we denote  a left zero semigroup and a right zero semigroup on a set $S$, respectively.   If $S$ is a set of cardinality $|S|=n$,  we use the notations $LO_n$ and $RO_n$ for representatives of the classes of semigroups isomorphic to $LO_S$ and $RO_S$, respectively.

\smallskip

A semigroup $(S, *)$ is called \emph{rectangular}~\cite{ZhCA2017a} if every element of $S$ is a middle identity, that is,
$x*y*z = x*z$ for all  $x, y, z \in S$. In other words, the product in $S$ depends only on the first and the last factors. 
A nontrivial null semigroup is an  example of a commutative rectangular semigroup that has no identity element.
In~\cite{Nagy}, it was proved that a semigroup $(S,*)$ is rectangular if and only if the factor semigroup $S/\theta$ is a right zero semigroup, 
where $\theta = \{(a, b)\in S\times S : x*a = x*b \text{ for all } x\in S\}$. The authors also described a method for constructing all rectangular semigroups. In~\cite{GS_sie}, the study focused on rectangular ideal extensions of left (right) zero semigroups by null quotients, and provided new combinatorial results for the numbers of pairwise nonisomorphic instances.

\smallskip

If $(S,*)$ is a semigroup, then the semigroup $(S,{*}^d)$ with operation $x{*}^d y=y* x$ is called {\em dual} to $(S,*)$, denoted $(S,*)^d$.
It is clear that $(S,*)^d = (S,*)$ holds precisely when $(S,*)$ is a commutative semigroup, that $\Aut(S,*)^d = \Aut(S,*)$, and that the dual of a rectangular semigroup is itself rectangular.

\section{Duality properties of doppelsemigroups}

In this section, we introduce several notions of duality for doppelsemigroups and study the relationships among them.

If $(D, \dashv, \vdash)$ is a doppelsemigroup and we define
$$x\dashv^d y = y \dashv x,\ \ \  x\vdash^d y = y \vdash x,$$
then it is straightforward to verify that $(D,\vdash, \dashv)$, $(D,\dashv^d, \vdash^d)$, $(D,\vdash^d, \dashv^d)$ are also doppelsemigroups. We refer to them respectively as the {\em opposite}, {\em dual}, and {\em cross-dual} doppelsemigroups of $(D, \dashv, \vdash)$, denoted respectively by $(D, \dashv, \vdash)^{op}$, $(D, \dashv, \vdash)^{d}$, and $(D, \dashv, \vdash)^{cd}$.

It follows that the unary operations 
$$op: (D, \dashv, \vdash) \mapsto (D, \dashv, \vdash)^{op},\quad  d: (D, \dashv, \vdash) \mapsto (D, \dashv, \vdash)^{d},$$ 
$$\text{ and }\ cd: (D, \dashv, \vdash) \mapsto (D, \dashv, \vdash)^{cd}$$ 
are involutive in the sense that
$$((D, \dashv, \vdash)^{op})^{op} = (D, \dashv, \vdash),\quad ((D, \dashv, \vdash)^{d})^{d} = (D, \dashv, \vdash),$$
$$\text{ and }\ ((D, \dashv, \vdash)^{cd})^{cd} = (D, \dashv, \vdash).$$

A doppelsemigroup $(D, \dashv, \vdash)$ is said to be {\em self-opposite}, {\em self-dual}, and {\em self-cross-dual} if
$$(D, \dashv, \vdash)^{op} = (D, \dashv, \vdash),\quad (D, \dashv, \vdash)^{d} = (D, \dashv, \vdash),\quad \text{ and }\ (D, \dashv, \vdash)^{cd} = (D, \dashv, \vdash),$$ respectively.

\smallskip

Since $(D, \dashv, \vdash)$ is self-opposite if and only if its two operations coincide, we obtain the following proposition.

\begin{proposition}\label{ds_trivial}
Let $(D,\dashv, \vdash)$ be a doppelsemigroup. Then $(D,\dashv, \vdash)$ is self-opposite if and only if $(D,\dashv, \vdash)$ is trivial.
\end{proposition}

\begin{corollary} The class of nontrivial doppelsemigroups decomposes into pairs of opposite doppelsemigroups. 
\end{corollary}

Since  $(S,*)^d$ coincides with $(S,*)$ precisely when it is commutative, we obtain the following proposition.

\begin{proposition}\label{ds_com}
Let $(D,\dashv, \vdash)$ be a doppelsemigroup. Then $(D,\dashv, \vdash)$ is self-dual if and only if $(D,\dashv, \vdash)$ is commutative.
\end{proposition}

\begin{corollary} The class of noncommutative doppelsemigroups decomposes into pairs of dual doppelsemigroups. 
\end{corollary}

A doppelsemigroup $(D,\dashv, \vdash)$ is called {\em abelian} if $x \dashv y = y \vdash x$ for all $x,y\in D$.

\begin{proposition}\label{ds_ab}
Let $(D,\dashv, \vdash)$ be a doppelsemigroup. Then the following conditions are equivalent:
\begin{itemize}
\item[1)] $(D,\dashv)$ and $(D,\vdash)$ are dual semigroups;
\item[2)] $(D,\dashv, \vdash)$ is abelian;
\item[3)] $(D,\dashv, \vdash)$ is self-cross-dual.
\end{itemize}
\end{proposition}

\begin{proof} $(1)\Rightarrow(2)$ If $(D,\dashv)$ and $(D,\vdash)$ are dual semigroups, then $x\dashv y = y\vdash x $ for all $x,y\in D$, and hence $(D,\dashv, \vdash)$ is an abelian doppelsemigroup.

$(2)\Rightarrow(3)$ Let $(D,\dashv, \vdash)$ be an abelian doppelsemigroup. Taking into account that 
$$x \vdash^d y = y \vdash x = x \dashv y\ \ \text{  and  }\ \  x \dashv^d y = y \dashv x = x \vdash y,$$
we conclude that $\vdash^d\ =\ \dashv$ and $\dashv^d\ =\ \vdash$, and hence $(D,\dashv, \vdash)$ is a self-cross-dual doppelsemigroup.

$(3)\Rightarrow(1)$ If $(D,\dashv, \vdash)$ is a self-cross-dual doppelsemigroup, then $x \dashv y = x \vdash^d y = y \vdash x$ for all $x,y\in D$, and hence $(D,\dashv)$ and $(D,\vdash)$ are dual semigroups.
\end{proof}

\begin{corollary} The class of nonabelian doppelsemigroups decomposes into pairs of cross-dual doppelsemigroups. 
\end{corollary}

Since commutative semigroups $(D,\dashv)$ and $(D,\vdash)$ are dual if and only if their operations coincide, Proposition~\ref{ds_ab} implies the following corollaries.

\begin{corollary}\label{com_na_ds} Commutative nontrivial doppelsemigroups are nonabelian.
\end{corollary}

\begin{corollary}\label{ab_ncom_ds} Abelian nontrivial doppelsemigroups are noncommutative.
\end{corollary}

\section{Connections between strong and commutative doppelsemigroups}

\begin{proposition}\label{ds_un_st}
Let $(D,\dashv, \vdash)$ be a doppelsemigroup. Then the following conditions are equivalent:
\begin{itemize}
\item[1)] $(D,\dashv, \vdash)$ is strong;
\item[2)] $(D,\dashv, \vdash)^{op}$ is strong;
\item[3)] $(D,\dashv, \vdash)^d$ is strong;
\item[4)] $(D,\dashv, \vdash)^{cd}$ is strong.
\end{itemize}
\end{proposition}
\begin{proof} The implication $(1)\Rightarrow(2)$ is trivial. 

$(2)\Rightarrow(3)$ Let $(D,\dashv, \vdash)^{op}$ be a strong doppelsemigroup, that is $$x\vdash(y\dashv z)=x\dashv (y\vdash z)\ \text{ for all } x,y,z\in D.$$ Taking into account that
\[
\begin{aligned}
x\dashv^d (y\vdash^d z) &= x\dashv^d (z\vdash y) = (z\vdash y)\dashv x = z\vdash (y\dashv x)  \\
&=  z\dashv (y\vdash x)=  (z\dashv y)\vdash x =  (y\dashv^d z)\vdash x = x\vdash^d(y\dashv^d z),
\end{aligned}
\]
for all $x,y,z\in D$, we conclude that $(D,\dashv, \vdash)^d = (D,\dashv^d, \vdash^d)$ is a strong doppelsemigroup as well.

Since $(D, \dashv, \vdash)^{cd}$ is the opposite of $(D, \dashv, \vdash)^d$, the implication $(3) \Rightarrow (4)$ follows directly from $(1) \Rightarrow (2)$.

$(4)\Rightarrow(1)$ Let $(D,\dashv, \vdash)^{cd} = (D,\vdash^d, \dashv^d)$ be a strong doppelsemigroup. Then $x\vdash^d (y\dashv^d z)= x\dashv^d(y\vdash^d z)$ for all $x,y,z\in D$.
It follows that
\[
\begin{aligned}
x \dashv (y \vdash z) &= x \dashv (z \vdash^d y) = (z \vdash^d y) \dashv^d x = z \vdash^d (y \dashv^d x)  \\
&= z \dashv^d (y \vdash^d x)= (z \dashv^d y) \vdash^d x = (y \dashv z) \vdash^d x = x \vdash (y \dashv z),
\end{aligned}
\]
for all $x,y,z\in D$, and thus $(D,\dashv, \vdash)$ is a strong doppelsemigroup as well.
\end{proof}

The following proposition shows that the class of strong doppelsemigroups contains the class of commutative doppelsemigroups.

\begin{proposition}\label{com_strong} Each commutative doppelsemigroup $(D,\dashv,\vdash)$ is strong.
\end{proposition}
\begin{proof}Indeed, for all $x,y,z\in D$,
\[
\begin{aligned}
x\dashv(y\vdash z) & = (y\vdash z)\dashv x = y\vdash (z\dashv x) = (z\dashv x)\vdash y \\
& = z\dashv (x\vdash y) = (x\vdash y)\dashv z = x\vdash (y\dashv z),
\end{aligned}
\]
and thus $(D,\dashv,\vdash)$ is a strong doppelsemigroup.
\end{proof}

Since commutativity of the operation $\dashv$ is equivalent to the commutativity of the operation $\dashv^d$, and commutativity of the operation $\vdash$ is equivalent to the commutativity of the operation $\vdash^d$, we obtain the following proposition.

\begin{proposition}\label{ds_un_com}
Let $(D,\dashv, \vdash)$ be a doppelsemigroup. Then the following conditions are equivalent:  
\begin{itemize}
\item[1)] $(D,\dashv, \vdash)$ is commutative;
\item[2)] $(D,\dashv, \vdash)^{op}$ is commutative;
\item[3)] $(D,\dashv, \vdash)^d$ is commutative;
\item[4)] $(D,\dashv, \vdash)^{cd}$ is commutative.
\end{itemize}
\end{proposition}

\section{Connections between rectangular and strong doppelsemigroups}

Observing that a semigroup $(S,*)$ is rectangular if and only if its dual is rectangular, we arrive at the following proposition.

\begin{proposition}\label{ds_un_rec}
Let $(D,\dashv, \vdash)$ be a doppelsemigroup. Then the following conditions are equivalent:
\begin{itemize}
\item[1)] $(D,\dashv, \vdash)$ is rectangular;
\item[2)] $(D,\dashv, \vdash)^{op}$ is rectangular;
\item[3)] $(D,\dashv, \vdash)^d$ is rectangular;
\item[4)] $(D,\dashv, \vdash)^{cd}$ is rectangular.
\end{itemize}
\end{proposition}

\smallskip

\begin{proposition}\label{strong_rect}
A rectangular doppelsemigroup is strong if and only if it is trivial.
\end{proposition}

\begin{proof}
Let $(D,\dashv,\vdash)$ be a rectangular doppelsemigroup.  
By the representation theorem for rectangular semigroups~\cite{Nagy}, the semigroup $(D,\dashv)$ decomposes as $D=\bigsqcup_{t\in T} S_t$, where $T$ is a right zero semigroup, and is equipped with maps $f_{t,r}:S_t\to S_r$ ($t,r\in T$) satisfying $f_{r,p}\circ f_{t,r}=f_{t,p}$; the product is given by $a\dashv b=f_{t,r}(a)$ for $a\in S_t$, $b\in S_r$.  
Similarly, $(D,\vdash)$ decomposes as $D=\bigsqcup_{u\in U} R_u$, with $U$ a right zero semigroup, maps $g_{u,v}:R_u\to R_v$ obeying $g_{v,w}\circ g_{u,v}=g_{u,w}$, and product $a\vdash b=g_{u,v}(a)$ for $a\in R_u$, $b\in R_v$.

Assume first that $(D,\dashv,\vdash)$ is strong, i.e., for all  $x,y,z\in D$,
\[
x\dashv (y\vdash z)=x\vdash (y\dashv z).
\]
Fix any $x\in D$. By the definition of a partition, there exist unique indices $t\in T$ and $u\in U$ such that $x\in S_t$ and $x\in R_u$.
For arbitrary $y,z\in D$ with $y\in S_r\cap R_v$ and $z\in S_p\cap R_w$, the rectangular representations yield
\[
x\dashv (y\vdash z)=x\dashv g_{v,w}(y)=f_{t,r'}(x), 
\qquad 
x\vdash (y\dashv z)=x\vdash f_{r,p}(y)=g_{u,v'}(x),
\]
for some indices $r'\in T$ and $v'\in U$ determined by the blocks of $y$ and $z$.  
Thus the strong condition implies
\[
f_{t,r'}(x)=g_{u,v'}(x).
\]

Taking $y=z$, we obtain, for all  $x,y\in D$,
\[
x\dashv (y\vdash y)=x\vdash (y\dashv y),
\]
which unfolds as
\[
f_{t,r''}(x)=g_{u,v''}(x),
\]
where $r''\in T$ and $v''\in U$ now depend only on the block of $y$.  
Since $\{S_r:r\in T\}$ and $\{R_v:v\in U\}$ are partitions of $D$, varying $y$ forces $r''$ to range over all of $T$ and $v''$ over all of $U$.  
Hence the above equality implies
\[
f_{t,r}(x)=g_{u,v}(x) 
\qquad\text{for all } t,r\in T,\ u,v\in U,\ \text{and all } x\in D.
\]

But the latter means precisely that the two operations coincide on every pair of elements. Indeed, for $x\in S_t\cap R_u$ and $y\in S_r\cap R_v$,
\[
x\dashv y = f_{t,r}(x)=g_{u,v}(x)=x\vdash y.
\]
Therefore, $(D,\dashv,\vdash)$ is trivial.

Conversely, if the doppelsemigroup is trivial, i.e.\ $x\dashv y=x\vdash y$ for all $x,y\in D$, then
\[
x\dashv (y\vdash z)=x\vdash (y\vdash z)=x\vdash (y\dashv z),
\]
so it is clearly strong.

Consequently, a rectangular doppelsemigroup is strong if and only if it is trivial.
\end{proof}

\begin{corollary}\label{rec_nstr_ds} Rectangular nontrivial doppelsemigroups are nonstrong.
\end{corollary}

\begin{corollary}\label{str_nrec_ds} Strong nontrivial doppelsemigroups are nonrectangular.
\end{corollary}

\section{Properties of doppelsemigroup isomorphisms}

A  map $\varphi : D_1 \to D_2$ is called a {\em homomorphism } from a doppelsemigroup $(D_1,\dashv_1, \vdash_1)$ to a doppelsemigroup $(D_2,\dashv_2, \vdash_2)$ if $$\varphi(a\dashv_1 b)=\varphi(a)\dashv_2\varphi(b)\ \ \text{  and  }\ \ \varphi(a\vdash_1 b)=\varphi(a)\vdash_2\varphi(b)$$ for all $a,b\in D_1$.

A bijective homomorphism $\psi : D_1 \to D_2$ is called an {\em isomorphism } from a doppelsemigroup $(D_1,\dashv_1, \vdash_1)$ to a doppelsemigroup $(D_2,\dashv_2, \vdash_2)$.
If there exists an isomorphism from a doppelsemigroup $(D_1,\dashv_1, \vdash_1)$ to a doppelsemigroup $(D_2,\dashv_2, \vdash_2)$, then $(D_1, \dashv_1, \vdash_1)$ and $(D_2, \dashv_2, \vdash_2)$ are said to be {\em isomorphic}, denoted $(D_1,\dashv_1, \vdash_1)\cong (D_2,\dashv_2, \vdash_2)$. An isomorphism $\psi: D\to D$ is called an {\em   automorphism} of a doppelsemigroup $(D,\dashv, \vdash)$. By $\Aut(D,\dashv, \vdash)$ we denote the automorphism group of a doppelsemigroup $(D,\dashv, \vdash)$. 

We define a doppelsemigroup $(D, \dashv, \vdash)$ to be {\em iso-opposite}, {\em iso-dual}, or {\em iso-cross-dual} if it is isomorphic to its opposite, dual, or cross-dual doppelsemigroup, respectively. Evidently, every trivial (self-opposite), commutative (self-dual), and abelian (self-cross-dual) doppelsemigroup is iso-opposite, iso-dual, and iso-cross-dual, respectively. However, the converse need not hold. Indeed, a doppelsemigroup may be iso-dual without being self-dual, since self-duality requires equality with the dual doppelsemigroup, whereas iso-duality requires only the existence of an isomorphism between them. Analogous remarks apply to the notions of self-opposite versus iso-opposite and self-cross-dual versus iso-cross-dual.

\begin{example} Consider any noncommutative group $(G, *)$ of even order and its variant $(G, *_a)$ generated by an element $a \in G$ of order $2$. Then $(G, *_a, *)$ forms a nontrivial noncommutative doppelsemigroup. Moreover, since 
$$a*_a*a= a*a*a = a^2*a = e*a =a \ne e = a*a,$$ the doppelsemigroup $(G, *_a, *)$ is nonabelian.

Let us show that the bijection $\psi: G \to G$ defined by $\psi(g) = a*g$ from the doppelsemigroup $(G, *_a, *)$ to its opposite dopplesemigroup $(G, *, *_a)$ is an isomorphism. Indeed,
\[
\begin{aligned}
& \psi(x*_a y) =  \psi(x*a*y) = a*x*a*y = \psi(x)*\psi(y)\ \text{ and } \\
& \psi(x*y) =  a*x*y  = a*x*a^2*y = (a*x)*a*(a*y) = \psi(x)*_a\psi(y).
\end{aligned}
\] 
Hence, $(G, *_a,*)$ is a nontrivial iso-opposite doppelsemigroup.

\medskip

Let us verify that the bijection $\phi: G \to G$ defined by $\phi(g) = g^{-1}$ from the doppelsemigroup $(G, *_a, *)$ to its dual dopplesemigroup $(G,  *_a^d, *^d)$ is an isomorphism. Indeed,

\[
\begin{aligned}
 \phi(x*_a y) &= \phi(x*a*y) = (x*a*y)^{-1} = y^{-1}*a^{-1}*x^{-1}  \\  
&= \phi(y)*a*\phi(x)=  \phi(y)*_a\phi(x) =  \phi(x)*_a^d\phi(y) \ \text{ and } \\
 \phi(x*y) &= (x*y)^{-1}  =y^{-1}*x^{-1} = \phi(y)*\phi(x) =  \phi(x)*^d\phi(y).
\end{aligned}
\] 
Therefore, $(G, *_a, *)$ is a noncommutative iso-dual doppelsemigroup.

\medskip

Finally, let us verify that the bijection $\varphi: G \to G$ defined by  $\varphi(g) = a*g^{-1}$ from the doppelsemigroup $(G, *_a, *)$ to its cross-dual dopplesemigroup $(G, *^d, *_a^d)$ is an isomorphism. Indeed,
\[
\begin{aligned}
\varphi(x*_a y) &= \varphi(x*a*y) = a*(x*a*y)^{-1} = a*y^{-1}*a^{-1}*x^{-1} \\
& = (a*y^{-1})*(a*x^{-1})= \varphi(y)*\varphi(x)= \varphi(x)*^d\varphi(y)\ \text{ and } \\
\varphi(x*y) &= a*(x*y)^{-1}  = a*y^{-1}*x^{-1} = a*y^{-1}*a^2*x^{-1} \\
& = (a*y^{-1})*a*(a*x^{-1})= \varphi(y)*a*\varphi(x)=\varphi(y)*_a\varphi(x) \\
& = \varphi(x)*_a^d\varphi(y).
\end{aligned}
\] 
Consequently, $(G, *_a, *)$ is a nonabelian iso-cross-dual doppelsemigroup.
\end{example}


\begin{proposition}\label{homo_ds_dual} Let $(D_1,\dashv_1, \vdash_1)$ and  $(D_2,\dashv_2, \vdash_2)$ be doppelsemigroups. Then for a map $\varphi: D_1 \to D_2$ the following conditions are equivalent:
\begin{itemize}
\item[1)] $\varphi$ is a homomorphism from $(D_1,\dashv_1, \vdash_1)$ to $(D_2,\dashv_2, \vdash_2)$;
\item[2)] $\varphi$ is a homomorphism from $(D_1,\dashv_1, \vdash_1)^{op}$ to $(D_2,\dashv_2, \vdash_2)^{op}$;
\item[3)] $\varphi$ is a homomorphism from $(D_1,\dashv_1, \vdash_1)^{d}$ to $(D_2,\dashv_2, \vdash_2)^{d}$;
\item[4)] $\varphi$ is a homomorphism from $(D_1,\dashv_1, \vdash_1)^{cd}$ to $(D_2,\dashv_2, \vdash_2)^{cd}$.
\end{itemize}
\end{proposition}

\begin{proof}  The implication $(1)\Rightarrow(2)$ is straightforward. 

\medskip

$(2)\Rightarrow(3)$ Let $\varphi$ be a homomorphism from $(D_1,\dashv_1, \vdash_1)^{op} = (D_1,\vdash_1, \dashv_1)$ to $(D_2,\dashv_2, \vdash_2)^{op}= (D_2,\vdash_2, \dashv_2)$.
Observing that
\[
\begin{aligned}
&\varphi(x\dashv_1^d y) = \varphi(y\dashv_1 x) = \varphi(y)\dashv_2 \varphi(x) = \varphi(x) \dashv_2^d \varphi(y) \text{ and} \\
&\varphi(x\vdash_1^d y) = \varphi(y\vdash_1 x) = \varphi(y)\vdash_2 \varphi(x) = \varphi(x) \vdash_2^d \varphi(y),
\end{aligned}
\]
for all $x,y\in D_1$, we deduce that $\varphi$ is a homomorphism from $(D_1,\dashv_1, \vdash_1)^{d}$ to $(D_2,\dashv_2, \vdash_2)^{d}$.

\medskip
Since $(D_1, \dashv_1, \vdash_1)^{cd}$ is the opposite of $(D_1, \dashv_1, \vdash_1)^d$ and $(D_2, \dashv_2, \vdash_2)^{cd}$ is the opposite of $(D_2, \dashv_2, \vdash_2)^d$, the implication $(3) \Rightarrow (4)$ is immediate.

\medskip
$(4)\Rightarrow(1)$ Let $\varphi$ be a homomorphism from $(D_1,\dashv_1, \vdash_1)^{cd} = (D_1,\vdash_1^d, \dashv_1^d)$ to $(D_2,\dashv_2, \vdash_2)^{cd}= (D_2,\vdash_2^d, \dashv_2^d)$.
Noting that
\[
\begin{aligned}
&\varphi(x\dashv_1 y) = \varphi(y\dashv_1^d x) = \varphi(y)\dashv_2^d \varphi(x) = \varphi(x) \dashv_2 \varphi(y) \text{ and} \\
&\varphi(x\vdash_1 y) = \varphi(y\vdash_1^d x) = \varphi(y)\vdash_2^d \varphi(x) = \varphi(x) \vdash_2 \varphi(y),
\end{aligned}
\]
for all $x,y\in D_1$, we conclude that $\varphi$ is a homomorphism from $(D_1,\dashv_1, \vdash_1)$ to $(D_2,\dashv_2, \vdash_2)$.
\end{proof}

\begin{corollary}\label{iso_ds_dual} Let $(D_1,\dashv_1, \vdash_1)$ and  $(D_2,\dashv_2, \vdash_2)$ be doppelsemigroups. Then the following conditions are equivalent:
\begin{itemize}
\item[1)] the doppelsemigroups $(D_1,\dashv_1, \vdash_1)$ and $(D_2,\dashv_2, \vdash_2)$ are isomorphic;
\item[2)] the doppelsemigroups $(D_1,\dashv_1, \vdash_1)^{op}$ and  $(D_2,\dashv_2, \vdash_2)^{op}$ are isomorphic;
\item[3)] the doppelsemigroups $(D_1,\dashv_1, \vdash_1)^{d}$ and  $(D_2,\dashv_2, \vdash_2)^{d}$ are isomorphic;
\item[4)] the doppelsemigroups $(D_1,\dashv_1, \vdash_1)^{cd}$ and  $(D_2,\dashv_2, \vdash_2)^{cd}$ are isomorphic.
\end{itemize}
\end{corollary}

\begin{corollary}\label{aut_ds_dual} Let $(D,\dashv, \vdash)$ be a doppelsemigroup. Then 
$$\Aut(D,\dashv, \vdash) = \Aut(D,\dashv, \vdash)^{op} = \Aut(D,\dashv, \vdash)^{d} = \Aut(D,\dashv, \vdash)^{cd}.$$ 
\end{corollary}

The following theorem was proved in~\cite{GDS_die}.

\begin{theorem}\label{char_iso_com_rec_ds} Let $\kappa$ be a cardinal, and let $D$ be a set with $|D| = \kappa$.
Up to isomorphism, there exists exactly one commutative rectangular doppelsemigroup of order $\kappa$, namely the trivial null doppelsemigroup $O_{D^0}$, whose automorphism group is $S_{D \setminus \{0\}}$.
\end{theorem}

\section{Refining the complete classification of nonisomorphic doppelsemigroups of order 3}\label{sec:3iso_clas}

In carrying out the classification, up to isomorphism, of all doppelsemigroups of order $3$ in \cite{GR1}, the nontrivial rectangular doppelsemigroups were overlooked.

\medskip

Let $R = \{a, b, c\}$. We define two associative binary operations, $\dashv$ and $\vdash$, on $R$ using the following Cayley tables:

$$
\begin{array}{||c||c|c|c||}
\hhline{|t:=:t:===:t|}
\dashv & a & b & c \\
\hhline{|:=::===:|}
a & a & a & a \\
\hhline{||-||---||}
b & b & b & b \\
\hhline{||-||---||}
c & a & a & a \\
\hhline{|b:=:b:===:b|}
\end{array}
\hspace{3cm}
\begin{array}{||c||c|c|c||}
\hhline{|t:=:t:===:t|}
\vdash & a & b & c \\
\hhline{|:=::===:|}
a & a & a & a \\
\hhline{||-||---||}
b & b & b & b \\
\hhline{||-||---||}
c & b & b & b \\
\hhline{|b:=:b:===:b|}
\end{array}
$$

\medskip

It is straightforward to verify that $(R, \dashv, \vdash)$ is a nontrivial noncommutative rectangular doppelsemigroup.

\smallskip
 
Consider the map $\psi: R \to R$ defined by $\psi(a) = b$, $\psi(b) = a$, and $\psi(c) = c$. One can easily verify that $\psi$ is an isomorphism from the doppelsemigroup $(R, \dashv, \vdash)$ to its opposite doppelsemigroup $(R, \dashv, \vdash)^{op} = (R, \vdash, \dashv)$. Consequently, $(R, \dashv, \vdash)$ is a nontrivial iso-opposite doppelsemigroup.

\smallskip

Because isomorphisms preserve left zero subdoppelsemigroups, and  $(R, \dashv, \vdash)$ has two left zeros whereas its dual $(R, \dashv, \vdash)^d$ has none (as it contains, by duality, two right zeros), the doppelsemigroups $(R, \dashv, \vdash)$ and  $(R, \dashv, \vdash)^d$ cannot be isomorphic. Since $(R, \dashv, \vdash)$ is isomorphic to its opposite  $(R, \dashv, \vdash)^{op}$, Corollary~\ref{iso_ds_dual} implies that its dual $(R, \dashv, \vdash)^d$ is likewise isomorphic to  $((R, \dashv, \vdash)^{op})^d = (R, \dashv, \vdash)^{cd}$. By Propositions~\ref{ds_un_com} and~\ref{ds_un_rec}, the dual structure  $(R, \dashv, \vdash)^d$ is a noncommutative rectangular doppelsemigroup. Moreover, Corollary~\ref{rec_nstr_ds} shows that neither $(R, \dashv, \vdash)$ nor its dual is strong. Additionally, Corollary~\ref{iso_ds_dual} ensures that $(R, \dashv, \vdash)^d$ is a nontrivial iso-opposite doppelsemigroup.

\smallskip

Taking into account the doppelsemigroups $(R, \dashv, \vdash)$ and $(R, \dashv, \vdash)^d$, together with the results from \cite{GR1} and the computer-aided calculations presented in Section~\ref{sec:ds-finite}, the following theorem provides a complete classification of all nonisomorphic doppelsemigroups of order $3$.

\begin{theorem}
Up to isomorphism, there exist $77$ pairwise non-isomorphic three-element doppelsemigroups, $41$ of which are commutative. Noncommutative doppelsemigroups form $18$ pairs of dual doppelsemigroups. Moreover, there are $65$ strong doppelsemigroups of order $3$. Furthermore, there are $7$ rectangular doppelsemigroups of order $3$, including one commutative trivial doppelsemigroup, two pairs of  noncommutative trivial dual doppelsemigroups and a single pair of noncommutative  iso-opposite nontrivial dual doppelsemigroups. Additionally, there exist exactly $24$ pairwise nonisomorphic three-element trivial doppelsemigroups.
\end{theorem}

\section{Number of doppelsemigroups of small order}\label{sec:ds-finite}

The determination of the numbers $\mathrm{s}(n)$ and $\mathrm{cs}(n)$, representing respectively the counts of all pairwise nonisomorphic semigroups of order 
$n$ and all pairwise nonisomorphic commutative semigroups of order $n$, is a difficult combinatorial problem. The functions $\mathrm{s}(n)$ and $\mathrm{cs}(n)$
grow very rapidly as $n$ tends to infinity. The sequences $(\mathrm{s}(n))$ and $(\mathrm{cs}(n))$ are listed in the On-Line Encyclopedia of Integer Sequences as entries A027851 and A001426, respectively. All currently known exact values of these sequences are presented in Tables~\ref{tab:sg} and~\ref{tab:csg}.

\medskip

\begin{table}[H]
\caption{Number of nonisomorphic semigroups up to order $9$}\label{tab:sg}
\medskip
    \centering
    \resizebox{13cm}{!}{
        \begin{tabular}{|r|c|c|c|c|c|c|c|c|c|c|c|c|}
            \hline
            $n$\ \ & 0 & 1  & 2 & 3 &   4 & 5  & 6 & 7 & 8 & 9  \\
            \cline{1-11}
            $\mathrm{s}(n)$ & 1 & 1 & 5 & 24 & 188 & 1915 & 28634 & 1627672 & 3684030417 & 105978177936292  \\
            \cline{1-11}
         \end{tabular}
    } 
\end{table}

\begin{table}[H]
\caption{Number of nonisomorphic commutative semigroups up to order $10$}\label{tab:csg}
\medskip
    \centering
    \resizebox{12.5cm}{!}{
        \begin{tabular}{|r|c|c|c|c|c|c|c|c|c|c|c|c|}
            \hline
            $n$\ \ & 0 & 1  & 2 & 3 &   4 & 5  & 6 & 7 & 8 & 9 & 10 \\
            \cline{1-12}
            $\mathrm{cs}(n)$ & 1 & 1 & 3 & 12 & 58 & 325 & 2143 & 17291 & 221805 & 11545843 & 3518930337 \\
            \cline{1-12}
         \end{tabular}
    }

\end{table}

Denote by $\mathrm{ds}(n)$, $\mathrm{cds}(n)$, $\mathrm{ads}(n)$, $\mathrm{sds}(n)$, and $\mathrm{rds}(n)$ the number of all pairwise non-isomorphic doppelsemigroups, commutative doppelsemigroups, abelian doppelsemigroups, strong doppelsemigroups, and rectangular doppelsemigroups of order $n$, respectively. We were able to calculate these cardinalities for small $n$. In Appendix~\ref{appnd}, we explain the method used to generate all pairwise nonisomorphic doppelsemigroups of order $n$ and provide a listing of the \texttt{Python} code employed for these computations. The results of (computer) calculations are presented in Tables~\ref{tab:ds}--\ref{tab:rds}.

\begin{table}[H]
\caption{Number of  nonisomorphic doppelsemigroups up to order 5}\label{tab:ds}
\medskip
    \centering
    \resizebox{6cm}{!}{
        \begin{tabular}{|r|c|c|c|c|c|c|c|}
            \hline
            $n$\ \ & 0 & 1  & 2 & 3 &   4 & 5   \\
            \cline{1-7}
            $\mathrm{ds}(n)$ & 1 & 1 & 8 & 77 & 1217 & 68177  \\
            \cline{1-7}
         \end{tabular}
    }
\end{table}

\begin{table}[H]
\caption{ Number of  nonisomorphic commutative doppelsemigroups up to order 6}\label{tab:cds}
\medskip
    \centering
    \resizebox{7.5cm}{!}{
        \begin{tabular}{|r|c|c|c|c|c|c|c|c|}
            \hline
            $n$\ \ & 0 & 1  & 2 & 3 &   4 & 5 & 6  \\
            \cline{1-8}
            $\mathrm{cds}(n)$ & 1 & 1 & 6 & 41 & 345 & 3892 &  134462 \\
            \cline{1-8}
        \end{tabular}
    }
\end{table}

Since a doppelsemigroup is considered trivial when its two operations coin\-cide, the numbers of all pairwise nonisomorphic nontrivial doppelsemigroups and  pairwise nonisomorphic  commutative nontrivial doppelsemigroups of order $n$ are equal to $\mathrm{ds}(n)\!-\!\mathrm{s}(n)$ and  $\mathrm{cds}(n)\!-\!\mathrm{cs}(n)$, respectively.

\begin{table}[H]
\caption{ Number of  nonisomorphic abelian doppelsemigroups up to order 6}\label{tab:ads}
\medskip
    \centering
    \resizebox{6.5cm}{!}{
        \begin{tabular}{|r|c|c|c|c|c|c|c|c|}
            \hline
            $n$\ \ & 0 & 1  & 2 & 3 &   4 & 5 & 6  \\
            \cline{1-8}
            $\mathrm{ads}(n)$ & 1 & 1 & 3 & 12 & 62 & 446 & 7503 \\
            \cline{1-8}
        \end{tabular}
    }
    
\end{table}

Since a commutative doppelsemigroup is abelian if and only if it is trivial, the number of all pairwise nonisomorphic trivial abelian doppelsemigroups of order $n$ is equal to $\mathrm{cs}(n)$, and hence the numbers of all pair\-wise nonisomorphic nonabelian commutative doppelsemigroups and pairwise nonisomorphic noncommutative abelian doppelsemigroups of order~$n$ are $\mathrm{cds}(n)\!-\!\mathrm{cs}(n)$ and $\mathrm{ads}(n)\!-\!\mathrm{cs}(n)$, respectively.

\begin{table}[H]
\caption{Number of  nonisomorphic strong doppelsemigroups up to order 5}\label{tab:sds}
\medskip
    \centering
    \resizebox{6cm}{!}{
        \begin{tabular}{|r|c|c|c|c|c|c|c|}
            \hline
            $n$\ \ & 0 & 1  & 2 & 3 &   4 & 5   \\
            \cline{1-7}
            $\mathrm{sds}(n)$ & 1 & 1 & 8 & 65 & 841 & 56428  \\
            \cline{1-7}
         \end{tabular}
    }
    
\end{table}

Proposition~\ref{com_strong} implies that the number of all pairwise nonisomorphic noncommutative strong doppelsemigroups  of order~$n$ is equal to $\mathrm{sds}(n)\!-\!\mathrm{cds}(n)$.

\begin{table}[H]
\caption{ Number of  nonisomorphic rectangular doppelsemigroups and semigroups up to order 6}\label{tab:rds}
\medskip
    \centering
    \resizebox{6cm}{!}{
        \begin{tabular}{|r|c|c|c|c|c|c|c|}
            \hline
            $n$\ \ & 0 & 1  & 2 & 3 &   4 & 5 & 6  \\
            \cline{1-8}
            $\mathrm{rds}(n)$ & 1 & 1 & 3 & 7 & 20 & 51 & 171 \\
            \cline{1-8}
            $\mathrm{rs}(n)$ & 1 & 1 & 3 & 5 & 10 & 14 & 27 \\
            \cline{1-8}
        \end{tabular}
    }
    
\end{table}

Table~\ref{tab:rds} also lists the numbers $\mathrm{rs}(n)$ of all pairwise non\-isomorphic rectangular semigroups for $n \leq 6$. Hence, the number of pairwise non\-isomorphic rectangular nontrivial doppelsemigroups of order~$n$ is given by $\mathrm{rds}(n)\!- \!\mathrm{rs}(n)$.

According to Theorem~\ref{char_iso_com_rec_ds}, the numbers of pairwise nonisomorphic noncommutative rectangular doppelsemigroups and pairwise nonisomorphic noncommutative rectangular nontrivial doppelsemigroups of order~$n$ are equal to $\mathrm{rds}(n)\!-\!1$ and $\mathrm{rds}(n)\!-\!\mathrm{rs}(n)$, respectively.

Since a rectangular doppelsemigroup is strong if and only if it is trivial, by Proposition~\ref{strong_rect}, the number of pairwise nonisomorphic nonrectangular strong doppelsemigroups of order~$n$ is $\mathrm{sds}(n)\!-\!\mathrm{rs}(n)$, while the number of pairwise nonisomorphic nonstrong rectangular doppelsemigroups of order~$n$ is $\mathrm{rds}(n)\!-\!\mathrm{rs}(n)$.

\begin{problem}Determine the numbers $\mathrm{ds}(n)$, $\mathrm{sds}(n)$  for $n\geq 6$ and $\mathrm{cds}(n)$, $\mathrm{ads}(n)$, $\mathrm{rds}(n)$ for $n\geq 7$.
\end{problem}


\appendix

\section{Program code for computing nonisomorphic dop\-pel\-semigroups}\label{appnd}

This code processes the Cayley tables of all nonisomorphic semigroups of order $n$ obtained with \texttt{GAP} using the \texttt{Smallsemi} library of semigroups of small order.  The corresponding  tables were relabeled from $\{1,\ldots,n\}$ to $\{0,\ldots,n\!-\!1\}$ for computational convenience and exported to \texttt{csv}-files. The code loads the Cayley tables from \texttt{csv}-files, generates all their permutations, and checks pairwise combinations for compliance with the system of doppelsemigroup axioms. It then eliminates isomorphic cases, retaining only representatives of isomorphism classes, and produces a complete list of pairwise nonisomorphic doppelsemigroups of a given order $n$. The results are saved as operation tables for $\dashv$ and $\vdash$ into a single \texttt{csv}-file for further analysis. For $n=6$, we translated this code from \texttt{Python} into \texttt{C++} and performed parallel computations. We present the \texttt{Python} version here, as it is more readable and accessible to the reader.

\begin{lstlisting}
import csv
import glob
import os
from itertools import permutations, product

# ----------------------------
# 1) Load semigroup tables
# ----------------------------
semigroup_tables = []
n = None  # size will be determined automatically

for filename in glob.glob("semigroup_*.csv"):
    with open(filename, newline='') as csvfile:
        reader = csv.reader(csvfile)
        table = []
        for row in reader:
            nums = [int(cell) for cell in row if cell.strip() != ""]
            if n is None:
                n = len(nums)  # determine size from the first row
            table.append(nums)
        semigroup_tables.append(table)

print(f"Loaded {len(semigroup_tables)} semigroup tables (order {n})")

# ----------------------------
# 2) Generate all permutations for each table
# ----------------------------
all_tables = []
for tab in semigroup_tables:
    for p in permutations(range(n)):
        inv = [0]*n
        for idx, val in enumerate(p):
            inv[val] = idx
        newtab = [[p[tab[inv[i]][inv[j]]] for j in range(n)] for i in range(n)]
        all_tables.append(newtab)

print(f"Total associative tables with permutations: {len(all_tables)}")

# ----------------------------
# 3) Check doppelsemigroup axioms
# ----------------------------
def is_doppelsemigroup(op_dashv, op_vdash):
	for x, y, z in product(range(n), repeat=3):
    	 if op_vdash[op_dashv[x][y]][z] != op_dashv[x][op_vdash[y][z]]: return False #D1
    	 if op_dashv[op_vdash[x][y]][z] != op_vdash[x][op_dashv[y][z]]: return False #D2
      	#if op_dashv[x][y] != op_dashv[y][x]: return False #CommL
      	#if op_vdash[x][y] != op_vdash[y][x]: return False #CommR
      	#if op_vdash[x][y] != op_dashv[y][x]: return False #Abelian
      	#if op_dashv[x][op_vdash[y][z]] != op_vdash[x][op_dashv[y][z]]: return False #Strong
      	#if op_dashv[x][op_dashv[y][z]] != op_dashv[x][z] or op_vdash[x][op_vdash[y][z]] != op_vdash[x][z]: return False #Rectangular
    return True

# ----------------------------
# 4) Find all doppelsemigroups
# ----------------------------
doppelsemigroup_pairs = []
for t1 in all_tables:
    for t2 in all_tables:
        if is_doppelsemigroup(t1, t2):
            doppelsemigroup_pairs.append((t1, t2))

print(f"Found {len(doppelsemigroup_pairs)} ordered doppelsemigroup pairs")

# ----------------------------
# 5) Canonization under permutations
# ----------------------------
canon_seen = set()
rep_pairs = []  # list of representatives to output
for t1, t2 in doppelsemigroup_pairs:
    keys = []
    for p in permutations(range(n)):
        inv = [0]*n
        for idx, val in enumerate(p):
            inv[val] = idx
        r1 = tuple(tuple(p[t1[inv[i]][inv[j]]] for j in range(n)) for i in range(n))
        r2 = tuple(tuple(p[t2[inv[i]][inv[j]]] for j in range(n)) for i in range(n))
        keys.append((r1, r2))
    min_key = min(keys)
    if min_key not in canon_seen:
        canon_seen.add(min_key)
        rep_pairs.append(min_key)

print(f"Number of nonisomorphic doppelsemigroups (order {n}): {len(canon_seen)}")

# ----------------------------
# 6) Save all nonisomorphic doppelsemigroups into one csv-file
# ----------------------------
out_dir = os.path.dirname(os.path.abspath(file))  # program directory
list_file = os.path.join(out_dir, f"list_gdm{n}.csv")

with open(list_file, "w", newline='', encoding="utf-8") as f:
    writer = csv.writer(f)
    for idx, (r1, r2) in enumerate(rep_pairs, start=1):
        writer.writerow([f"doppelsemigroup #{idx}"])
        writer.writerow(["Operation dashv"])
        writer.writerows(r1)
        writer.writerow([])  # empty line
        writer.writerow(["Operation vdash"])
        writer.writerows(r2)
        writer.writerow([])  # separator
        writer.writerow([])

print(f"\nAll nonisomorphic doppelsemigroups saved to {list_file}")
\end{lstlisting}

\end{document}